\documentclass[global,a4paper]{article}%
\usepackage{amsfonts}
\usepackage{amsmath}
\usepackage{amssymb}
\usepackage{graphicx}%
\newtheorem{theorem}{Theorem}

\newtheorem{lemma}[theorem]{Lemma}

\newtheorem{remark}[theorem]{Remark}

\newenvironment{proof}[1][Proof]{\noindent\textbf{#1.} }{\ \rule{0.5em}{0.5em}}
\begin{document}

\title{Korovkin type theorem and iterates of certain positive linear opeartors}
\author{N. I. Mahmudov\\Department of Mathematics\\Eastern Mediterranean University \\Gazimagusa, TRNC via Mersin 10, Turkey \\email: nazim.mahmudov@emu.edu.tr}
\maketitle

\begin{abstract}
In this paper we prove Korovkin type theorem for iterates of general positive
linear operators $T:C\left[  0,1\right]  \rightarrow C\left[  0,1\right]  $
and derive quantitative estimates in terms of modulus of smoothness. In
particular, we show that under some natural conditions the iterates
$T^{m}:C\left[  0,1\right]  \rightarrow C\left[  0,1\right]  $ converges
strongly to a fixed point of the original operator $T$. The results can be
applied to several well-known operators; we present here the $q$-MKZ
operators, the $q$-Stancu operators, the genuine $q$-Bernstein--Durrmeyer
operators and the Cesaro operators.

\end{abstract}

\textbf{Keywords.} Iterates of operators, degree of approximation,
$K$-functionals, modulus of smoothness, Bernstein operators, genuine
Bernstein--Durrmeyer operators, Stancu operators, Korovkin type theorem,
Cesaro operators, Meyer-K\H{o}nig and Zeller operators.

\section{Introduction}

These iterated Bernstein operators were investigated in the 60's and 70's by
P. C. Sikkema \cite{sikkema}, R. P. Kelisky \& T. J. Rivlin \cite{kelisky}, S.
Karlin \& Z. Ziegler \cite{karlin}, J. Nagel \cite{nagel1}, M. R. da Silva
\cite{dasilva1} and Gonska \cite{gonska1}, \cite{gonska2}. Some of this
research was later generalized by Altomare et al. (see, for example,
\cite{Alt1}, \cite{Alt2}, \cite{Alt3}). Altomare suggested to use in this
context an approach described by Dickmeis and Nessel \cite{nessel}. This was
done recently by Rasa in \cite{rasa1} and \cite{rasa2}. Another new paper
related to the subject of this article was written by S. Ostrovska
\cite{ostr1} on iterates of $q$-Bernstein polynomials.

The methods employed to study the convergence of iterates of some operators
occurring in Approximation Theory include Matrix Theory methods, like
stochastic matrices \cite{niel}, \cite{dasilva2}, \cite{wenz}, Korovkin-type
theorems \cite{karlin}, quantitative results about the approximation of
functions by positive linear operators \cite{gonska3}, \cite{gonska4}, fixed
point theorems \cite{agratini}, \cite{gonska5}, \cite{rus}, or methods from
the theory of $C_{0}$-semigroups, like Trotter's approximation theorem
\cite{karlin}, \cite{micchelli2}. However, these techniques fail to work for
the Meyer-K\H{o}nig and Zeller (MKZ) or the May operators. Very recently, I.
Gavrea and M. Ivan \cite{GI} proved that the iterates of the MKZ operates
converges strongly to $P\left(  f;x\right)  =\left(  1-x\right)  f\left(
0\right)  +xf\left(  1\right)  $. Once such convergence have been obtained,
the following natural question is to ask for rates of convergence. In Section
3, as a consequence of our results, we obtain the quantitaive estimates for
the iterates of the $q$-MKZ ($0<q\leq1$) operators, which is completely new.

On the other hand, because of its powerful applications, Korovkin's result has
been extended in many directions. There is an extensive literature on
Korovkin-type theorems, which may have had a summit already about twenty five
years ago. In particular, there exist abstract results that cover many
naturally arising concrete cases. The contributions up to about 1994 are
excellently documented in the book of Altomare and Campiti \cite{Alt3}. More
recent results obtained in \cite{wang2}, \cite{mah2} cover also approximation
of $q$-type operators.

In this paper we establish quantitative Korovkin type theorem for the iterates
of certain positive linear operators $T:C\left[  0,1\right]  \rightarrow
C\left[  0,1\right]  .$ As a consequence of our results, we obtain the
quantitaive estimates for the iterates of almost all classical and new
positive linear operators: the $q$-MKZ operators, the $q$-Stancu operators,
the genuine $q$-Bernstein--Durrmeyer operators in the case $0<q\leq1$ and the
Cesaro operators. It is worth mentioning that for $q=1$ these operators become
classical MKZ, Stancu and genuine Bernstein-Durrmeyer operators.

\section{Main results}

The following notations will be used throughout this paper. The classical
Petree's $K$-functional and the second modulus of smoothness of a function $f$
are defined respectively by
\[
K_{2}\left(  f,t\right)  :=\inf_{g\in C^{2}\left[  0,1\right]  }\left\{
\left\Vert f-g\right\Vert +t\left\Vert g^{\prime\prime}\right\Vert \right\}
\]
and%
\[
\omega_{2}(f,t):=\sup_{0<h\leq t}\sup_{0\leq x\leq1-2h}\left\vert
f(x+2h)-2f(x+h)+f(x)\right\vert .
\]
It is known there exists a constant $C>0$ such that%
\begin{equation}
K_{2}\left(  f,t\right)  \leq C\omega_{2}(f,\sqrt{t}). \label{k1}%
\end{equation}
Let $e_{i}:\left[  0,1\right]  \rightarrow R$ be the monomial functions
$e_{i}\left(  x\right)  =x^{i},$ $i=0,1,2.$

Now we formulate the main results of the paper. First result shows that under
the conditions (\ref{cc101}) the iterates of $T:C\left[  0,1\right]
\rightarrow C\left[  0,1\right]  $ converges to some linear positive operator
$T^{\infty}:C\left[  0,1\right]  \rightarrow C\left[  0,1\right]  .$

\begin{theorem}
\label{t:main} Suppose that $T:C\left[  0,1\right]  \rightarrow C\left[
0,1\right]  $ is a positive linear operator such that
\begin{align}
T\left(  e_{0}\right)   &  =e_{0},\ \ T\left(  e_{1};x\right)  \leq
x,\ \ \nonumber\\
\lim_{n\rightarrow\infty}\left\Vert T^{m}\left(  e_{1}\right)  -T^{\infty
}\left(  e_{1}\right)  \right\Vert  &  =\lim_{n\rightarrow\infty}\left\Vert
T^{m}\left(  e_{2}\right)  -T^{\infty}\left(  e_{2}\right)  \right\Vert =0.
\label{cc101}%
\end{align}
Then there exists a linear positive operator $T^{\infty}:C\left[  0,1\right]
\rightarrow C\left[  0,1\right]  $ such that the following pointwise estimate%
\begin{equation}
\left\vert \left(  T^{m}-T^{\infty}\right)  \left(  f;x\right)  \right\vert
\leq k\omega_{2}\left(  f,\sqrt{\lambda_{n}\left(  x\right)  }\right)
+k\left\Vert f\right\Vert \delta_{n}\left(  x\right)  \label{ss1}%
\end{equation}
holds true \textit{for }$x\in\left[  0,1\right]  $\textit{ and }$f\in C\left[
0,1\right]  $\textit{, where }$k$ is an absolute constant and%
\begin{align*}
\lambda_{n}\left(  x\right)   &  =\max\left\{  \left\vert \left(
T^{m}-T^{\infty}\right)  \left(  e_{1};x\right)  \right\vert ,\left\vert
\left(  T^{\infty}-T^{m}\right)  \left(  e_{2};x\right)  \right\vert \right\}
,\\
\delta_{n}\left(  x\right)   &  =\left\vert \left(  T^{m}-T^{\infty}\right)
\left(  e_{1};x\right)  \right\vert .
\end{align*}

\end{theorem}

\begin{proof}
For every nonincreasing convex $g\in C^{2}\left[  0,1\right]  ,$ we have
\begin{align}
g\left(  t\right)   &  \geq g\left(  x\right)  +g^{\prime}\left(  x\right)
\left(  t-x\right)  ,\nonumber\\
T\left(  g;x\right)   &  \geq g\left(  x\right)  +g^{\prime}\left(  x\right)
\left(  T\left(  e_{1};x\right)  -x\right)  \geq g\left(  x\right)  .
\label{ss2}%
\end{align}
It follows that
\[
g\left(  x\right)  \leq T^{m}\left(  g;x\right)  \leq T^{m+1}\left(
g;x\right)  .
\]
In other words the sequence $\left\{  T^{m}\left(  g;x\right)  \right\}  $ is
nondecreasing for any nonincreasing convex $g\in C^{2}\left[  0,1\right]  $
and $x\in\left[  0,1\right]  $.

Let $x\in\left[  0,1\right]  $ be fixed and $g\in C^{2}\left[  0,1\right]  $
be arbitrary. Introduce the following auxiliary functions%
\[
g_{\pm}\left(  t\right)  =\frac{1}{2}\left\Vert g^{\prime\prime}\right\Vert
\left(  1-t\right)  ^{2}+\left\Vert g^{\prime}\right\Vert \left(  1-t\right)
\pm g\left(  t\right)  .
\]
It is clear that
\[
g_{\pm}^{\prime}\left(  t\right)  =-\left\Vert g^{\prime\prime}\right\Vert
\left(  1-t\right)  -\left\Vert g^{\prime}\right\Vert \pm g\left(  t\right)
\leq0,\ \ \ g_{\pm}^{\prime\prime}\left(  t\right)  =\left\Vert g^{\prime
\prime}\right\Vert \pm g\left(  t\right)  \geq0.
\]
Therefore the functions $g_{\pm}\left(  t\right)  $ are nonincreasing convex
for both choices of the sign. Since $\left(  T^{m+p}-T^{m}\right)  \left(
g_{\pm};x\right)  $ is positive we have%
\begin{align*}
0  &  \leq\left(  T^{m+p}-T^{m}\right)  \left(  g_{\pm};x\right)  =\frac{1}%
{2}\left\Vert g^{\prime\prime}\right\Vert \left(  T^{m+p}-T^{m}\right)
\left(  \left(  e_{0}-e_{1}\right)  ^{2};x\right) \\
&  +\left\Vert g^{\prime}\right\Vert \left(  T^{m+p}-T^{m}\right)  \left(
e_{0}-e_{1};x\right)  \pm\left(  T^{m+p}-T^{m}\right)  \left(  g;x\right)  .
\end{align*}
It follows that%
\begin{equation}
\left\vert \left(  T^{m+p}-T^{m}\right)  \left(  g;x\right)  \right\vert
\leq\frac{1}{2}\left\Vert g^{\prime\prime}\right\Vert \left\vert \left(
T^{m+p}-T^{m}\right)  \left(  e_{2};x\right)  \right\vert +\left(  \left\Vert
g^{\prime\prime}\right\Vert +\left\Vert g^{\prime}\right\Vert \right)
\left\vert \left(  T^{m}-T^{m+p}\right)  \left(  e_{1};x\right)  \right\vert .
\label{ss3}%
\end{equation}
So $\left\{  T^{m}\left(  f;x\right)  \right\}  $ is a Cauchy sequence in
$C\left[  0,1\right]  $ and there is a linear positive operator $T^{\infty
}\left(  f\right)  $ such that%
\[
\lim_{m\rightarrow\infty}\left\Vert T^{m}\left(  f\right)  -T^{\infty}\left(
f\right)  \right\Vert =0
\]
for any $f\in C\left[  0,1\right]  .$ Taking the limit as $p\rightarrow\infty$
in (\ref{ss3}) and using the well known inequality%
\[
\left\Vert g^{\prime}\right\Vert \leq C_{1}\left(  \left\Vert g\right\Vert
+\left\Vert g^{\prime\prime}\right\Vert \right)
\]
we have%
\begin{align}
\left\vert \left(  T^{\infty}-T^{m}\right)  \left(  g;x\right)  \right\vert
&  \leq\frac{1}{2}\left\Vert g^{\prime\prime}\right\Vert \left\vert \left(
T^{\infty}-T^{m}\right)  \left(  e_{2};x\right)  \right\vert +\left(
\left\Vert g^{\prime\prime}\right\Vert +\left\Vert g^{\prime}\right\Vert
\right)  \left\vert \left(  T^{m}-T^{\infty}\right)  \left(  e_{1};x\right)
\right\vert \nonumber\\
&  \leq\left(  \frac{3}{2}+C_{1}\right)  \lambda_{n}\left(  x\right)
\left\Vert g^{\prime\prime}\right\Vert +C_{1}\delta_{n}\left(  x\right)
\left\Vert g\right\Vert . \label{ss4}%
\end{align}
Taking into account that $\left\Vert T^{m}\right\Vert =1$ from (\ref{ss4}) and
the inequality $\left\Vert g\right\Vert \leq\left\Vert f\right\Vert
+\left\Vert f-g\right\Vert $ it follows that%
\begin{align*}
\left\vert T^{\infty}\left(  f;x\right)  -T^{m}\left(  f;x\right)
\right\vert  &  \leq\left\vert \left(  T^{\infty}-T^{m}\right)  \left(
f-g;x\right)  \right\vert +\left\vert T^{\infty}\left(  g;x\right)
-T^{m}\left(  g;x\right)  \right\vert \\
&  \leq2\left\Vert f-g\right\Vert +\left(  \frac{3}{2}+C_{1}\right)
\lambda_{n}\left(  x\right)  \left\Vert g^{\prime\prime}\right\Vert
+C_{1}\delta_{n}\left(  x\right)  \left\Vert g\right\Vert \\
&  \leq C_{2}\left(  \left\Vert f-g\right\Vert +\lambda_{n}\left(  x\right)
\left\Vert g^{\prime\prime}\right\Vert \right)  +C_{4}\delta_{n}\left(
x\right)  \left\Vert f\right\Vert
\end{align*}
Taking on the right side the infimum over all $g\in C^{2}[0,1]$ we obtain
\[
\left\vert T^{\infty}\left(  f;x\right)  -T^{m}\left(  f;x\right)  \right\vert
\leq C_{2}K_{2}\left(  f;\frac{1}{4}\left\vert \left(  T^{\infty}%
-T^{m}\right)  \left(  e_{2};x\right)  \right\vert \right)  +C_{4}\delta
_{n}\left(  x\right)  \left\Vert f\right\Vert .
\]
Now using (\ref{k1}) we obtain (\ref{ss1}).
\end{proof}

\begin{remark}
It is clear that for any $f\in C\left[  0,1\right]  $ the the image
$T^{\infty}\left(  f\right)  $ is a fixed point of the original operator
$T:C\left[  0,1\right]  \rightarrow C\left[  0,1\right]  .$ It gives some
information on the nature of the limit operator $T^{\infty}$.
\end{remark}

If $T:C\left[  0,1\right]  \rightarrow C\left[  0,1\right]  $ preserves the
affine functions as a consequence of the above theorem we have the following result.

\begin{theorem}
\label{thm1} Suppose that $T:C\left[  0,1\right]  \rightarrow C\left[
0,1\right]  $ is a positive linear operator such that
\begin{equation}
T\left(  e_{0}\right)  =e_{0},\ \ T\left(  e_{1};x\right)  =x,\ \ \ \lim
_{n\rightarrow\infty}\left\Vert T^{m}\left(  e_{2}\right)  -T^{\infty}\left(
e_{2}\right)  \right\Vert =0. \label{cc10}%
\end{equation}
Then there exists a linear positive operator $T^{\infty}:C\left[  0,1\right]
\rightarrow C\left[  0,1\right]  $ such that the following pointwise estimate%
\begin{equation}
\left\vert \left(  T^{m}-T^{\infty}\right)  \left(  f;x\right)  \right\vert
\leq k\omega_{2}\left(  f,\frac{1}{2}\sqrt{\left\vert \left(  T^{\infty}%
-T^{m}\right)  \left(  e_{2};x\right)  \right\vert }\right)  \label{rr1}%
\end{equation}
holds true \textit{for }$x\in\left[  0,1\right]  $\textit{ and }$f\in C\left[
0,1\right]  $\textit{, where }$k$ is an absolute constant.
\end{theorem}

Next result shows that under the conditions (\ref{ccc1}) the limit $T^{\infty
}$ of the iterates $T^{m}$ is exactly the operator $P\left(  f\right)
:=\left(  e_{0}-e_{1}\right)  f\left(  0\right)  +e_{1}f\left(  1\right)  $
with the quantitative estimate (\ref{est5}).

\begin{theorem}
\label{t:iterate}Let $T:C\left[  0,1\right]  \rightarrow C\left[  0,1\right]
$ be a positive linear operator such that%
\begin{equation}
T\left(  e_{0}\right)  =e_{0},\ \ T\left(  e_{1};x\right)  =x,\ T\left(
e_{2};x\right)  \leq ax^{2}+bx,a,b\in R\backslash\left\{  0\right\}  ,\ a+b=1.
\label{ccc1}%
\end{equation}
Then the pointwise approximation
\begin{equation}
\left\vert T^{m}\left(  f;x\right)  -P\left(  f;x\right)  \right\vert \leq
k\omega_{2}\left(  f;\sqrt{a^{m}x\left(  1-x\right)  }\right)  \label{est5}%
\end{equation}
holds true for all $x\in\left[  0,1\right]  $ and $f\in C\left[  0,1\right]  $.
\end{theorem}

\begin{proof}
By the assumptions we have
\[
x^{2}\leq T\left(  e_{2};x\right)  \leq ax^{2}+bx\leq\left(  a+b\right)
x=x=T\left(  e_{1};x\right)  .
\]
On the other hand by the induction we have
\[
x^{2}\leq T^{m}\left(  e_{2};x\right)  \leq a^{m}x^{2}+b\left(
1+a+...+a^{m-1}\right)  x=a^{m}x^{2}+\left(  1-a^{m}\right)  x.
\]
So
\begin{align*}
0  &  \leq x-T^{m}\left(  e_{2};x\right)  \leq a^{m}x\left(  1-x\right)  ,\\
\lim_{m\rightarrow\infty}\left\Vert T^{m}\left(  e_{2}\right)  -e_{1}%
\right\Vert  &  =0,
\end{align*}
and the operator $T^{\infty}$ of Theorem \ref{thm1} satisfies%
\[
T^{\infty}\left(  e_{0}\right)  =e_{0},\ \ T^{\infty}\left(  e_{1}\right)
=e_{1},\ \ \ T^{\infty}\left(  e_{2}\right)  =e_{1}%
\]
and%
\[
\left(  T^{\infty}-T^{m}\right)  \left(  e_{2};x\right)  \leq a^{m}x\left(
1-x\right)  .
\]
It remains to show that $T^{\infty}\left(  f\right)  =P\left(  f\right)  $ for
all $f\in C\left[  0,1\right]  $. It is clear that it is enough to show this
equality in $C^{2}\left[  0,1\right]  $. Let $g\in C^{2}\left[  0,1\right]  $.
Define the following auxiliary functions.%
\[
G\left(  x\right)  :=g\left(  x\right)  -P\left(  g;x\right)  ,\ \ \ l:=\frac
{1}{2}\left\Vert G^{\prime\prime}\right\Vert =\frac{1}{2}\left\Vert
g^{\prime\prime}\right\Vert ,\ \ g_{\pm}\left(  x\right)  :=-lx^{2}+lx\pm
G\left(  x\right)  .
\]
It is clear that $g_{\pm}$ is concave and nonnegative, since
\[
g_{\pm}^{\prime\prime}\left(  x\right)  =-\left\Vert G^{\prime\prime
}\right\Vert \pm G^{\prime\prime}\left(  x\right)  \leq0,\ \ \ G\left(
0\right)  =G\left(  1\right)  =0.
\]
It follows that
\[
-l\left(  x-x^{2}\right)  \leq G\left(  x\right)  \leq l\left(  x-x^{2}%
\right)  ,\ \ \ \ 0\leq x\leq1.
\]
Applying the positive operator $T^{\infty}$ we get%
\[
-l\left(  T^{\infty}\left(  e_{1};x\right)  -T^{\infty}\left(  e_{2};x\right)
\right)  \leq T^{\infty}\left(  G;x\right)  =T^{\infty}\left(  g;x\right)
-P\left(  g;x\right)  \leq l\left(  T^{\infty}\left(  e_{1};x\right)
-T^{\infty}\left(  e_{2};x\right)  \right)  ,\ \ \ \
\]
for all $0\leq x\leq1,$ and consequently%
\[
T^{\infty}\left(  g\right)  =P\left(  g\right)
\]
for all $g\in C^{2}\left[  0,1\right]  $, which completes the proof.
\end{proof}

\begin{remark}
It is worth mentioning that the conditions $T\left(  e_{0}\right)
=e_{0},\ \ T\left(  e_{1}\right)  =e_{1},\ T\left(  e_{2};x\right)
=ax^{2}+bx$ are satisfied by the many classical positive linear operators
defined on $C\left[  0,1\right]  .$ The condition $T\left(  e_{2};x\right)
\leq ax^{2}+bx,,a,b\in R\backslash\left\{  0\right\}  ,\ a+b=1,$ covers the
$q$-MKZ operators.
\end{remark}

The last result shows that under the conditions (\ref{cc111}) the limit
$T^{\infty}$ of the iterates $T^{m}$ is exactly the operator $f\left(
0\right)  $ with the quantitative estimate (\ref{est51}).

\begin{theorem}
\label{t:iterate2}Let $T:C\left[  0,1\right]  \rightarrow C\left[  0,1\right]
$ be a positive linear operator such that%
\begin{equation}
T\left(  e_{0}\right)  =e_{0},\ \ T\left(  e_{1};x\right)  \leq x,\ \ \lim
_{m\rightarrow\infty}\left\Vert T^{m}\left(  e_{1}\right)  \right\Vert
=\lim_{m\rightarrow\infty}\left\Vert T^{m}\left(  e_{2}\right)  \right\Vert
=0. \label{cc111}%
\end{equation}
Then the pointwise approximation%
\begin{equation}
\left\vert T^{m}\left(  f;x\right)  -f\left(  0\right)  \right\vert \leq
k\omega_{2}\left(  f,\sqrt{\lambda_{m}\left(  x\right)  }\right)  +k\left\Vert
f\right\Vert \delta_{m}\left(  x\right)  \label{est51}%
\end{equation}
holds true \textit{for }$x\in\left[  0,1\right]  $\textit{ and }$f\in C\left[
0,1\right]  $\textit{, where }$k$ is an absolute constant and%
\begin{align*}
\lambda_{m}\left(  x\right)   &  =\max\left\{  T^{m}\left(  e_{1};x\right)
,T^{m}\left(  e_{2};x\right)  \right\}  ,\\
\delta_{m}\left(  x\right)   &  =T^{m}\left(  e_{1};x\right)  .
\end{align*}

\end{theorem}

\begin{proof}
The operator $T^{\infty}$ of Theorem \ref{thm1} satisfies%
\[
T^{\infty}\left(  e_{0}\right)  =e_{0},\ \ T^{\infty}\left(  e_{1}\right)
=0,\ \ \ T^{\infty}\left(  e_{2}\right)  =0.
\]
It remains to show that $T^{\infty}\left(  f\right)  =f\left(  0\right)  $ for
all $f\in C\left[  0,1\right]  $. It is clear that it is enough to show this
equality in $C^{2}\left[  0,1\right]  $. Let $g\in C^{2}\left[  0,1\right]  $.
Define the following auxiliary functions.%
\begin{align*}
G\left(  x\right)   &  :=g\left(  x\right)  -g\left(  0\right)  ,\ \ \ l_{2}%
:=\frac{1}{2}\left\Vert G^{\prime\prime}\right\Vert =\frac{1}{2}\left\Vert
g^{\prime\prime}\right\Vert ,\ \ l_{1}:=l_{2}+\left\vert g\left(  1\right)
-g\left(  0\right)  \right\vert ,\\
\ \ g_{\pm}\left(  x\right)   &  :=-l_{2}x^{2}+l_{1}x\pm G\left(  x\right)  .
\end{align*}
It is clear that $g_{\pm}$ is concave and nonnegative, since
\[
g_{\pm}^{\prime\prime}\left(  x\right)  =-\left\Vert G^{\prime\prime
}\right\Vert \pm G^{\prime\prime}\left(  x\right)  \leq0,\ \ \ g_{\pm}\left(
0\right)  =0,\ \ g_{\pm}\left(  1\right)  =-l_{2}+l_{1}\pm G\left(  1\right)
=\left\vert g\left(  1\right)  -g\left(  0\right)  \right\vert \pm G\left(
1\right)  \geq0.\
\]
It follows that
\[
-l_{1}x+l_{2}x^{2}\leq G\left(  x\right)  \leq l_{1}x-l_{2}x^{2},\ \ \ \ 0\leq
x\leq1.
\]
Applying the linear positive operator $T^{\infty}$ we get%
\[
-l_{1}T^{\infty}\left(  e_{1};x\right)  +l_{2}T^{\infty}\left(  e_{2}%
;x\right)  \leq T^{\infty}\left(  G;x\right)  =T^{\infty}\left(  g;x\right)
-g\left(  0\right)  \leq l_{1}T^{\infty}\left(  e_{1};x\right)  -l_{2}%
T^{\infty}\left(  e_{2};x\right)  ,\ \ \ \ 0\leq x\leq1,
\]
and consequently%
\[
T^{\infty}\left(  g;x\right)  -g\left(  0\right)  =0.
\]
for all $g\in C^{2}\left[  0,1\right]  $, which completes the proof.
\end{proof}

Many of the linear methods of approximation are given by a sequence of
discrete linear positive operators designed as follows.%
\begin{equation}
\Lambda_{n}\left(  f;x\right)  :=\sum_{k=0}^{n}f\left(  x_{n,k}\right)
a_{n,k}\left(  x\right)  ,\ \ \ f\in C\left[  0,1\right]  , \label{d1}%
\end{equation}
where every function $a_{n,k}\in C\left[  0,1\right]  $ is non-negative,
$0=x_{n,0}<...<x_{n,n}=1$ forms a mesh of nodes. Assume that the following
identities%
\begin{align}
\sum_{k=0}^{n}a_{n,k}\left(  x\right)   &  =1,\ \ \ \sum_{k=0}^{n}%
x_{n,k}a_{n,k}\left(  x\right)  =x,\ \ \ 0\leq x\leq1,\nonumber\\
a_{n,k}\left(  x\right)   &  \geq0,\ \ \ a_{n,0}\left(  0\right)
=a_{n,1}\left(  1\right)  =1,\ \ \label{d2}%
\end{align}
are fulfilled. It is clear that
\[
\Lambda_{n}\left(  f;0\right)  =f\left(  0\right)  ,\ \ \Lambda_{n}\left(
f;1\right)  =f\left(  1\right)  .
\]
Iterates of the discrete operators $\Lambda_{n}$ was studied by O. Agratini
and I. A.Rus \cite{AR} via contraction principle. In the next theorem we give
uniform estimation for the iterates of $\Lambda_{n}$

\begin{theorem}
\label{t:k}Let $\Lambda_{n}$ be defined by (\ref{d1}) and satisfies
(\ref{d2}). Assume that $u_{n}:=\min\left\{  a_{n,0}\left(  x\right)
+a_{n,n}\left(  x\right)  :0\leq x\leq1\right\}  >0$. Then the pointwise
approximation
\[
\left\vert \Lambda_{n}^{m}\left(  f;x\right)  -P\left(  f;x\right)
\right\vert \leq k\omega_{2}\left(  f;\sqrt{\left\vert e_{1}\left(  x\right)
-\Lambda_{n}^{m}\left(  e_{2};x\right)  \right\vert }\right)
\]
holds true for all $x\in\left[  0,1\right]  $ and $f\in C\left[  0,1\right]
$. Furthermore, we have the following uniform estimation%
\[
\left\Vert \Lambda_{n}^{m}\left(  f\right)  -P\left(  f\right)  \right\Vert
\leq k\omega_{2}\left(  f;\sqrt{\left(  1-u_{n}\right)  ^{m}\left\Vert
\Lambda_{n}\left(  e_{2}\right)  -e_{1}\right\Vert }\right)  .
\]

\end{theorem}

\begin{proof}
For each $0\leq x\leq1$ we can write%
\begin{align*}
\left\vert \Lambda_{n}^{m+1}\left(  e_{2};x\right)  -x\right\vert  &
=\left\vert \sum_{k=0}^{n}\left(  \Lambda_{n}^{m}\left(  e_{2};x_{n,k}\right)
-x_{n,k}\right)  a_{n,k}\left(  x\right)  \right\vert \leq\sum_{k=0}%
^{n}\left\vert \Lambda_{n}^{m}\left(  e_{2};x_{n,k}\right)  -x_{n,k}%
\right\vert a_{n,k}\left(  x\right) \\
&  \leq\sum_{k=1}^{n-1}\left\vert \Lambda_{n}^{m}\left(  e_{2};x_{n,k}\right)
-x_{n,k}\right\vert a_{n,k}\left(  x\right)  \leq\left\vert 1-a_{n,0}\left(
x\right)  -a_{n,n}\left(  x\right)  \right\vert \left\Vert \Lambda_{n}%
^{m}\left(  e_{2}\right)  -e_{1}\right\Vert \\
&  \leq\left(  1-u_{n}\right)  \left\Vert \Lambda_{n}^{m}\left(  e_{2}\right)
-e_{1}\right\Vert .
\end{align*}
By the induction we have%
\[
\left\Vert \Lambda_{n}^{m+1}\left(  e_{2}\right)  -e_{1}\right\Vert
\leq\left(  1-u_{n}\right)  ^{m}\left\Vert \Lambda_{n}\left(  e_{2}\right)
-e_{1}\right\Vert .
\]
If $u_{n}:=\min\left\{  a_{n,0}\left(  x\right)  +a_{n,n}\left(  x\right)
:0\leq x\leq1\right\}  >0$ then
\[
\lim_{m\rightarrow\infty}\left\Vert \Lambda_{n}^{m+1}\left(  e_{2}\right)
-e_{1}\right\Vert =0.
\]

\end{proof}

\begin{remark}
Actually, $\Lambda_{n}$ is a wide class of discrete operators that include
Bernstein-Sheffer (see P. Sablonniere \cite{sab}) , Stancu operators and
Cheney-Sharma operators
\end{remark}

\section{Applications}

In this section we employ the standard notations of $q$-calculus. $q$-integer
and $q$-factorial are defined by%
\begin{align*}
\left[  n\right]  _{q}  &  :=\left\{
\begin{array}
[c]{c}%
\dfrac{1-q^{n}}{1-q}\ \ \ \ \text{if\ \ \ }q\in R^{+}\backslash\{1\},\\
n\ \ \ \ \ \ \ \text{if\ \ \ \ \ }q=1
\end{array}
\right.  \ \ \ \ \ \text{for }n\in N\ \ \ \ \text{and\ \ \ }\left[  0\right]
_{q}=0,\\
\left[  n\right]  _{q}!  &  :=\left[  1\right]  _{q}\left[  2\right]
_{q}...\left[  n\right]  _{q}\ \ \ \ \text{for }n\in N\ \ \ \ \text{and\ \ \ }%
\left[  0\right]  _{q}!=1.
\end{align*}
For integers $0\leq k\leq n$ $q$-binomial is defined by%
\[
\left[
\begin{array}
[c]{c}%
n\\
k
\end{array}
\right]  _{q}:=\frac{\left[  n\right]  _{q}!}{\left[  k\right]  _{q}!\left[
n-k\right]  _{q}!}.
\]
In this section we apply the main result of the paper to discuss the limit of
the iterates of a special class of operators.

\subsection{Iterates of the $q$-MKZ operators ($0<q\leq1$)}

$q$-MKZ operators $M_{n,q}:C\left[  0,1\right]  \rightarrow C\left[
0,1\right]  ,$ $n\in\mathbb{N}$, are defined by%

\[
M_{n,q}\left(  f;x\right)  =\left\{
\begin{tabular}
[c]{ll}%
$\left(  1-x\right)  _{q}^{n+1}%
{\displaystyle\sum\limits_{k=0}^{\infty}}
f\left(  \frac{\left[  k\right]  _{q}}{\left[  n+k\right]  _{q}}\right)
\left[
\begin{array}
[c]{c}%
n+k\\
k
\end{array}
\right]  _{q}x^{k},$ & $0\in x<1,$\\
$f\left(  1\right)  ,$ & $x=1.$%
\end{tabular}
\ \ \ \ \ \right.
\]
MKZ operators were introduced by Meyer-K\"{o}nig and Zeller, $q$-MKZ operators
by T.Trif. It is worth mentioning that the second moment $M_{n,q}\left(
e_{2}\right)  $ of the $q$-MKZ operators cannot be expressed as a finite
combination of elementary functions since this moment turns out to be a
generalized hypergeometric function. This was a major barrier in calculating
the limit of the iterates of the $q$-MKZ operators. Very recently, I. Gavrea
and M. Ivan \cite{GI} proved that the iterates of the MKZ operates converges
strongly to $P\left(  f;x\right)  =\left(  1-x\right)  f\left(  0\right)
+xf\left(  1\right)  $. In what follows we solve this problem for $q$-MKZ
operators using Korovkin type theorem for the iterates of the positive linear
operators. Furthermore, we give the quantitative estimate for the iterates of
the $q$-MKZ operators, which is completely new.

\begin{lemma}
$M_{n,q}\left(  e_{2};x\right)  $ satisfies the condition%
\[
M_{n,q}\left(  e_{2};x\right)  \leq\left(  1-\frac{1}{\left[  n+1\right]
_{q}}\right)  x^{2}+\frac{1}{\left[  n+1\right]  _{q}}x.
\]

\end{lemma}

\begin{proof}
It is easy to see that%
\begin{align*}
M_{n,q}\left(  e_{2};x\right)  -x^{2}  &  =x\left(  1-x\right)  _{q}^{n+1}%
{\displaystyle\sum_{k=0}^{\infty}}
\left(  \frac{\left[  k+1\right]  _{q}}{\left[  n+k+1\right]  _{q}}%
-\frac{\left[  k\right]  _{q}}{\left[  n+k\right]  _{q}}\right)  \left[
\begin{array}
[c]{c}%
n+k\\
k
\end{array}
\right]  _{q}x^{k}\\
&  =x%
{\displaystyle\sum_{k=0}^{\infty}}
\left(  \frac{\left[  k+1\right]  _{q}}{\left[  n+k+1\right]  _{q}}%
-\frac{\left[  k\right]  _{q}}{\left[  n+k\right]  _{q}}\right)
m_{n,k}\left(  q;x\right) \\
&  =x%
{\displaystyle\sum_{k=0}^{\infty}}
\frac{\left[  k+1\right]  _{q}\left[  n+k\right]  _{q}-\left[  k\right]
_{q}\left[  n+k+1\right]  _{q}}{\left[  n+k+1\right]  _{q}\left[  n+k\right]
_{q}}m_{n,k}\left(  q;x\right) \\
&  =x%
{\displaystyle\sum_{k=0}^{\infty}}
\frac{q^{k}\left[  n+k\right]  _{q}-\left[  k\right]  _{q}q^{n+k}}{\left[
n+k+1\right]  _{q}\left[  n+k\right]  _{q}}m_{n,k}\left(  q;x\right) \\
&  =x%
{\displaystyle\sum_{k=0}^{\infty}}
\frac{q^{k}\left(  \left[  n\right]  _{q}+q^{n}\left[  k\right]  _{q}\right)
-\left[  k\right]  _{q}q^{n+k}}{\left[  n+k+1\right]  _{q}\left[  n+k\right]
_{q}}m_{n,k}\left(  q;x\right) \\
&  =x%
{\displaystyle\sum_{k=0}^{\infty}}
\frac{q^{k}\left[  n\right]  _{q}}{\left[  n+k+1\right]  _{q}\left[
n+k\right]  _{q}}m_{n,k}\left(  q;x\right) \\
&  =x%
{\displaystyle\sum_{k=0}^{\infty}}
\frac{q^{k}\left[  n\right]  _{q}}{\left[  n+k+1\right]  _{q}\left[
n+k\right]  _{q}}m_{n,k}\left(  q;x\right) \\
&  =x\left(  1-x\right)  _{q}^{n+1}%
{\displaystyle\sum_{k=0}^{\infty}}
\frac{q^{k}}{\left[  n+k+1\right]  _{q}}\left[
\begin{array}
[c]{c}%
n+k-1\\
k
\end{array}
\right]  _{q}x^{k}%
\end{align*}
It follows that%
\begin{align*}
M_{n,q}\left(  e_{2};x\right)   &  =x^{2}+x\left(  1-x\right)  _{q}^{n+1}%
{\displaystyle\sum\limits_{k=0}^{\infty}}
\frac{\left(  qx\right)  ^{k}}{\left[  n+k+1\right]  _{q}}\left[
\begin{array}
[c]{c}%
n+k-1\\
k
\end{array}
\right]  _{q}\\
&  =x^{2}+\frac{x\left(  1-x\right)  }{\left[  n+1\right]  _{q}}\left(
1-qx\right)  _{q}^{n}%
{\displaystyle\sum\limits_{k=0}^{\infty}}
\frac{\left[  n+1\right]  _{q}}{\left[  n+k+1\right]  _{q}}\left[
\begin{array}
[c]{c}%
n+k-1\\
k
\end{array}
\right]  _{q}\left(  qx\right)  ^{k}\\
&  \leq x^{2}+\frac{x\left(  1-x\right)  }{\left[  n+1\right]  _{q}}=\left(
1-\frac{1}{\left[  n+1\right]  _{q}}\right)  x^{2}+\frac{1}{\left[
n+1\right]  _{q}}x\leq x.
\end{align*}

\end{proof}

Thus it is obvious that $M_{n,q}$ satisfies the requirements of Theorem
\ref{t:iterate}. We arrive at the following theorem.

\begin{theorem}
Let $0<q\leq1.$ Let $M_{n,q}$ be a sequence of $q$-MKZ operators. Then the
pointwise approximation
\[
\left\vert M_{n,q}^{m}\left(  f;x\right)  -P\left(  f;x\right)  \right\vert
\leq k\omega_{2}\left(  f;\sqrt{\left(  1-\frac{1}{\left[  n+1\right]  _{q}%
}\right)  ^{m}x\left(  1-x\right)  }\right)
\]
holds true for all $x\in\left[  0,1\right]  $ and $f\in C\left[  0,1\right]  $.
\end{theorem}

\subsection{Iterates for the Cesaro operators}

Define the Cesaro operator $C:C\left[  0,1\right]  \rightarrow C\left[
0,1\right]  $%
\[
C\left(  f;x\right)  =\left\{
\begin{tabular}
[c]{ll}%
$f\left(  0\right)  ,$ & $x=0,$\\
$\frac{1}{x}\int_{0}^{x}f\left(  s\right)  ds,$ & $0<x<1,$\\
$\int_{0}^{1}f\left(  s\right)  ds,$ & $x=1.$%
\end{tabular}
\ \ \right.
\]
Simple calculations show that%
\begin{align*}
C\left(  e_{0};x\right)   &  =1,\ \ \ C\left(  e_{1};x\right)  =\frac{x}%
{2},\ \ \ C\left(  e_{2};x\right)  =\frac{x^{2}}{3},\ \\
C^{m}\left(  e_{0};x\right)   &  =1,\ \ \ C^{m}\left(  e_{1};x\right)
=\frac{x}{2^{m}},\ \ \ C^{m}\left(  e_{2};x\right)  =\frac{x^{2}}{3^{m}}.
\end{align*}
Hence an application of Theorem \ref{t:iterate2} yields the following statement.

\begin{theorem}
Let $C:C\left[  0,1\right]  \rightarrow C\left[  0,1\right]  $ be a Cesaro
operator. Then the pointwise approximation%
\[
\left\vert T^{m}\left(  f;x\right)  -f\left(  0\right)  \right\vert \leq
k\omega_{2}\left(  f,\sqrt{\lambda_{m}\left(  x\right)  }\right)  +k\left\Vert
f\right\Vert \delta_{m}\left(  x\right)
\]
holds true \textit{for }$x\in\left[  0,1\right]  $\textit{ and }$f\in C\left[
0,1\right]  $\textit{, where }$k$ is an absolute constant and%
\[
\lambda_{m}\left(  x\right)  =\max\left\{  \frac{x}{2^{m}},\frac{x^{2}}{3^{m}%
}\right\}  ,\ \ \ \delta_{m}\left(  x\right)  =\frac{x}{2^{m}}.
\]
\bigskip
\end{theorem}

\subsection{Iterates of the genuine $q$-Bernstein-Durrmeyer operators
($0<q\leq1$)}

We consider now the genuine $q$-Bernstein--Durrmeyer operators
\begin{equation}
U_{n,q}\left(  f;x\right)  :=f\left(  0\right)  p_{n,0}\left(  q;x\right)
+f\left(  1\right)  p_{n,n}\left(  q;x\right)  +\left[  n-1\right]  _{q}%
{\displaystyle\sum_{k=1}^{n-1}}
q^{1-k}p_{n,k}\left(  q;x\right)  \int_{0}^{1}p_{n-2,k-1}\left(  q;qt\right)
f\left(  t\right)  d_{q}t, \label{bd}%
\end{equation}
studied in \cite{ms}. Classical genuine Bernstein--Durrmeyer operators
appeared first in papers W. Chen \cite{chen} and T.N.T. Goodman and A. Sharma
\cite{GSH}.

In case of the genuine $q$-Bernstein--Durrmeyer operators we have%

\begin{align*}
U_{n,q}\left(  e_{0};x\right)   &  =1,\ \ \ \ \ U_{n,q}\left(  e_{1};x\right)
=x,\\
U_{n,q}\left(  e_{2};x\right)   &  =\left(  1-\frac{\left[  2\right]  _{q}%
}{\left[  n+1\right]  _{q}}\right)  x^{2}+\frac{\left[  2\right]  _{q}%
}{\left[  n+1\right]  _{q}}x,\\
U_{n,q}^{m}\left(  e_{2};x\right)   &  =\left(  1-\frac{\left[  2\right]
_{q}}{\left[  n+1\right]  _{q}}\right)  ^{m}x^{2}+\left[  1-\left(
1-\frac{\left[  2\right]  _{q}}{\left[  n+1\right]  _{q}}\right)  ^{m}\right]
x.
\end{align*}
Hence an application of Theorem \ref{t:iterate} yields the following statement.

\begin{theorem}
Let $0<q\leq1$. Let $U_{n,q}$ be a sequence of genuine $q$%
-Bernstein--Durrmeyer operators. Then the pointwise approximation
\[
\left\vert U_{n,q}^{m}\left(  f;x\right)  -P\left(  f;x\right)  \right\vert
\leq k\omega_{2}\left(  f;\sqrt{\left(  1-\frac{\left[  2\right]  _{q}%
}{\left[  n+1\right]  _{q}}\right)  ^{m}x\left(  1-x\right)  }\right)
\]
holds true for all $x\in\left[  0,1\right]  $ and $f\in C\left[  0,1\right]  $.
\end{theorem}

\subsection{Iterates of the $q$-Stancu operators ($0<q\leq1$)}

We apply now the above results for iterates of the $q$-Stancu operators%
\begin{align*}
S_{n,q}^{\left\langle \alpha,\beta,\gamma\right\rangle }  &  :C\left[
0,1\right]  \ni f\rightarrow%
{\displaystyle\sum\limits_{k=0}^{n}}
f\left(  \frac{\left[  k\right]  _{q}+\left[  \beta\right]  _{q}}{\left[
n\right]  _{q}+\left[  \gamma\right]  _{q}}\right)  p_{n,k}\left(
q;\alpha;\cdot\right)  \in P_{n},\ \ \\
p_{n,k}\left(  q;\alpha;x\right)   &  =\left[
\begin{array}
[c]{c}%
n\\
k
\end{array}
\right]  _{q}\frac{%
{\displaystyle\prod\nolimits_{i=0}^{k-1}}
\left(  x+\alpha\left[  i\right]  _{q}\right)
{\displaystyle\prod\nolimits_{s=0}^{n-k-1}}
\left(  x-q^{s}x+\alpha\left[  s\right]  _{q}\right)  }{%
{\displaystyle\prod\nolimits_{i=0}^{n-1}}
\left(  1+\alpha\left[  i\right]  _{q}\right)  }.
\end{align*}
$q$-Stancu operators $S_{n,q}^{\left\langle \alpha,0,0\right\rangle }$ were
introduced and studied in \cite{Nowak}. Let $0<q\leq1$ and $\alpha\geq0.$ Then
we have%
\begin{align*}
S_{n,q}^{\left\langle \alpha,0,0\right\rangle }\left(  1;x\right)   &
=1,\ \ \ S_{n,q}^{\left\langle \alpha,0,0\right\rangle }\left(  t;x\right)
=x,\\
\left(  S_{n,q}^{\left\langle \alpha,0,0\right\rangle }\right)  ^{m}\left(
t^{2};x\right)   &  =\left(  \left(  1-\frac{1}{\left[  n\right]  _{q}%
}\right)  \frac{1}{1+\alpha}\right)  ^{m}x^{2}+\left[  1-\left(  \left(
1-\frac{1}{\left[  n\right]  _{q}}\right)  \frac{1}{1+\alpha}\right)
^{m}\right]  x.
\end{align*}
It is worthwhile to mention that already in 1978 G. Mastroianni \& M. R.
Occorsio \cite{mas} have introduced and investigated the iterates of
$S_{n}^{\left\langle \alpha,0,0\right\rangle }$ ($q=1$) by extending a
procedure used by R. P. Kelisky \& T. J. Rivlin for the Bernstein operators.

$q$-Stancu operators $S_{n,q}^{\left\langle \alpha,\beta,\gamma\right\rangle
}$ were introduced and studied in \cite{mah11}. Let $\alpha=\beta
=0,\ \gamma>0.$ Then it is easy to show%
\begin{align*}
S_{n,q}^{\left\langle 0,0,\gamma\right\rangle }\left(  e_{0};x\right)   &
=1,\ \ \ S_{n,q}^{\left\langle 0,0,\gamma\right\rangle }\left(  e_{1}%
;x\right)  =\frac{\left[  n\right]  _{q}}{\left[  n\right]  _{q}+\left[
\gamma\right]  _{q}}x,\\
S_{n,q}^{\left\langle 0,0,\gamma\right\rangle }\left(  e_{2};x\right)   &
=\frac{\left[  n\right]  _{q}^{2}-\left[  n\right]  _{q}}{\left(  \left[
n\right]  _{q}+\left[  \gamma\right]  _{q}\right)  ^{2}}x^{2}+\frac{\left[
n\right]  _{q}}{\left(  \left[  n\right]  _{q}+\left[  \gamma\right]
_{q}\right)  ^{2}}x,\\
\left(  S_{n,q}^{\left\langle 0,0,\gamma\right\rangle }\right)  ^{m}\left(
e_{1};x\right)   &  =\left(  \frac{\left[  n\right]  _{q}}{\left[  n\right]
_{q}+\left[  \gamma\right]  _{q}}\right)  ^{m}x,\\
\left(  S_{n,q}^{\left\langle 0,0,\gamma\right\rangle }\right)  ^{m}\left(
e_{2};x\right)   &  =\left(  \frac{\left[  n\right]  _{q}^{2}-\left[
n\right]  _{q}}{\left(  \left[  n\right]  _{q}+\left[  \gamma\right]
_{q}\right)  ^{2}}\right)  ^{m}x^{2}\\
&  +\frac{\left[  n\right]  _{q}^{m}}{\left(  \left[  n\right]  _{q}+\left[
\gamma\right]  _{q}\right)  ^{m+1}}\sum_{i=0}^{m-1}\left(  \frac{\left[
n\right]  _{q}-1}{\left[  n\right]  _{q}+\left[  \gamma\right]  _{q}}\right)
^{i}x.
\end{align*}
We arrive at the following theorem.

\begin{theorem}
Let $0<q\leq1.$ Let $S_{n,q}^{\left\langle \alpha,\beta,\gamma\right\rangle }$
be a sequence of $q$-Stancu operators.

\begin{enumerate}
\item If $\alpha\geq0,$ $\beta=\gamma=0$ then the pointwise approximation
\[
\left\vert \left(  S_{n,q}^{\left\langle \alpha,0,0\right\rangle }\right)
^{m}\left(  f;x\right)  -P\left(  f;x\right)  \right\vert \leq k\omega
_{2}\left(  f;\sqrt{\left(  \left(  1-\frac{1}{\left[  n\right]  _{q}}\right)
\frac{1}{1+\alpha}\right)  ^{m}x\left(  1-x\right)  }\right)
\]
holds true for all $x\in\left[  0,1\right]  $ and $f\in C\left[  0,1\right]  $.

\item If $\alpha=\beta=0,\ \gamma>0$ then%
\[
\left\vert \left(  S_{n,q}^{\left\langle 0,0,\gamma\right\rangle }\right)
^{m}\left(  f;x\right)  -f\left(  0\right)  \right\vert \leq k\omega
_{2}\left(  f,\sqrt{\lambda_{m}\left(  x\right)  }\right)  +k\left\Vert
f\right\Vert \delta_{m}\left(  x\right)
\]
holds true \textit{for }$x\in\left[  0,1\right]  $\textit{ and }$f\in C\left[
0,1\right]  $\textit{, where }$k$ is an absolute constant and%
\begin{align*}
\lambda_{m}\left(  x\right)   &  =\max\left\{  \left(  S_{n,q}^{\left\langle
0,0,\gamma\right\rangle }\right)  ^{m}\left(  e_{1};x\right)  ,\left(
S_{n,q}^{\left\langle 0,0,\gamma\right\rangle }\right)  ^{m}\left(
e_{2};x\right)  \right\}  ,\\
\delta_{m}\left(  x\right)   &  =\left(  S_{n,q}^{\left\langle 0,0,\gamma
\right\rangle }\right)  ^{m}\left(  e_{1};x\right)  .
\end{align*}

\end{enumerate}
\end{theorem}

\end{document}